\documentclass[oneside,a4paper]{amsart}

\usepackage{etex}
\usepackage{xcolor}
\usepackage[utf8]{inputenc}
\usepackage[T1]{fontenc}
\makeatletter
\def\alloc@#1#2#3#4#5%
{\ifnum\count1#1<#4
	\allocationnumber\count1#1
	\global\advance\count1#1\@ne
	\global#3#5\allocationnumber
	\wlog{\string#5=\string#2\the\allocationnumber}%
	\else\ifnum#1<6
	\def\etex@dummy@definition{}
	\begingroup \escapechar\m@ne
	\expandafter\alloc@@\expandafter{\string#2}#5%
	\else\errmessage{No room for a new #2}\fi\fi
}
\makeatother
\usepackage[thinlines,thiklines]{easybmat}
\usepackage{amsfonts}

\usepackage[all,cmtip]{xy}
\usepackage{amsmath}
\usepackage{amsthm}
\usepackage{amssymb}
\usepackage{enumitem}
\usepackage{booktabs}
\usepackage{tikz}
\usepackage{tikz-cd}
\usetikzlibrary{calc,arrows,backgrounds,decorations.markings, shapes.misc}
\usepackage{longtable}
\usepackage{float}
\usepackage{floatflt}

\usepackage[backref=page]{hyperref}
\usepackage{hyperref}
\hypersetup{pdfauthor={Cong Ding}}
\hypersetup{colorlinks=true, linkcolor=red,anchorcolor=red,citecolor=blue}

\theoremstyle{plain}

\newtheorem{theorem}{\bf Theorem}[section]
\newtheorem{definition}[theorem]{\bf Definition}

\newtheorem{lemma}[theorem]{\bf Lemma}
\newtheorem{proposition}[theorem]{\bf Proposition}
\newtheorem{prop-def}[theorem]{\bf Proposition-Definition}

\newtheorem{remark}[theorem]{\bf Remark}

\newtheorem*{theorem*}{Theorem}

\newcommand{\im}{{\text {Im}}}
\newcommand{\Span}{{\text {span}}}
\newcommand{\Aut}{{\text {Aut}}}
\newcommand{\Pic}{{\text {Pic}}}

\begin{document}
		\tikzset{->-/.style={decoration={
				markings,
				mark=at position .5 with {\arrow{>}}},postaction={decorate}}}
	
	\tikzset{cross/.style={cross out, draw=black, minimum size=15*(#1-\pgflinewidth), inner sep=0pt, outer sep=0pt},
		cross/.default={1pt}}
	\title[Rigidity of admissible pairs]{On the rigidity of admissible pairs of irreducible Hermitian symmetric spaces which are of deletion type}
	\author{Cong Ding}
	\address{Morningside Center of Mathematics, Academy of Mathematics \& Systems Science, The Chinese Academy of Sciences, Beijing, China}
	\email{congding@amss.ac.cn}
	\date{\today}
	\maketitle
		\begin{abstract} 
			The notion of admissible pairs of rational homogeneous spaces of Picard number one and their rigidity in terms of the geometric substructures was formulated by Mok and Zhang. The rigidity is known for subdiagram type. While when the admissible pair is of deletion type, the rigidity no longer holds and we need additional conditions. Mok gave a general criterion to handle the admissible pairs which are not of subdiagram type.
			In this short note we give some examples of admissible pairs of irreducible compact Hermitian symmetric spaces of deletion type to be rigid under the assumption of rational saturation, as an application of Mok's theorem.
		\end{abstract}
\section{Introduction}
Mok and Zhang \cite{MR3960268} introduced the notion of admissible pairs of rational homogeneous spaces of Picard number one and their rigidity in terms of the geometric substructures. In \cite{MR3960268}, it was concluded that the admissible pair $(X_0, X)$ of subdiagram type is rigid when $X_0$ is not linear, by checking the non-degeneracy of substructures and doing parallel transport on the VMRT. However, according to the classification of admissible pairs in Hermitian type given by Zhang \cite{Zhang2014}, there are admissible pairs which are not of subdiagram type, a typical example is $(Q^m, Q^n)$ embedded in a standard way. In those cases, non-rigidity appears since the non-degeneracy condition does not hold, see for example \cite{Zhang2014}, \cite{Zhang2015}. One reason for the failure of rigidity is that a germ $S\subset X$ with substructure modeled on $(X_0, X)$  may not be saturated by open subsets of minimal rational curves of $X$. In \cite{MR4028278}, Mok gives a sufficient condition for admissible pairs of Hermitian type to be rigid, under the natural additional assumption of rational saturation. More precisely, he proved the following.
\begin{theorem}[Theorem 3.1 in \cite{MR4028278}]\label{Mokthmsaturation}
	Let $(X_0, X)$ be an admissible pair of irreducible Hermitian symmetric space of compact type with rank $\geq 2$, $0\subset X_0\subset X$ is the base point. Suppose that 
	\begin{enumerate}
		\item
	the pair of VMRTs $(\mathcal{C}_0(X_0), \mathcal{C}_0(X))$ satisfies Condition (F);
	
	\item the restriction map $r: \Gamma(X, T_X)\rightarrow \Gamma(X_0, T_{X|X_0})$ is surjective; 
	\end{enumerate}
Let $W\subset X$ be an open subset in the complex topology and assume that $S\subset W$ inherits by tangential intersection a sub-VMRT structure $\overline{\omega}: \mathcal{C}(S)\rightarrow S$ modeled on $(X_0, X)$, where $\mathcal{C}(S):=\mathcal{C}(X)\cap \mathbb{P}T(S)$. Assume more that $S\subset W$ is rational saturated, i.e. any germ of minimal rational curve $\ell$ of $X$ passing $0$ with $T_0(\ell)\subset T_0(S)$ must lie in $S$. Then, there exists $\gamma \in \text{Aut}(X)$ such that $S$ is an open subset of $\gamma(X_0)$.
\end{theorem}
The theorem provides sufficient conditions to check the rigidity under the assumption of rational saturation.	In \cite{MR4028278}, examples $(Q^m, Q^n)$, $(G^{III}(n,n), G(n,n))$ are checked. In this note, we furtherly check a class of examples satisfying the two conditions in Theorem \ref{Mokthmsaturation}, where $(X_0 ,X)$ is of deletion type (see Definition \ref{deletion}, from the classification given in \cite{Zhang2014}) and maximal. We have

\begin{theorem}\label{thmdeletion}
		Let $(X_0, X)$ be an admissible pair of irreducible Hermitian symmetric spaces of compact type which are of deletion type. Moreover assume that $(X_0, X)$ is maximal, i.e. it is not a composition of two admissible pairs of deletion type. Then two conditions in Theorem \ref{Mokthmsaturation} are satisfied and $(X_0, X)$ is rigid under the assumption of rational saturation.
\end{theorem}

When $(X_0, X)$ is not maximal, we can find counterexamples violating the Condition (F), so these sufficient conditions may not always work. The checking of Condition (F) depends on the fact that $(\mathcal{C}_0(X_0), \mathcal{C}_0(X))$ is of deletion type if $(X_0, X)$ is of deletion type. The checking of the surjectivity of the restriction map may need the Borel-Weil-Bott theorem in general, however we have a simple observation in the case when $(X_0, X)$ is maximal.

The note is organized as follows. In Section \ref{preli} we recall some prelimiaries; In Section \ref{sketch} we give a sketch of Mok's proof of Theorem \ref{Mokthmsaturation}. In Section \ref{checkingconditionF} and Section \ref{checkingsujectivity} we check the two conditions in Theorem \ref{Mokthmsaturation}
respectively. In section \ref{counterexample} we give a counterexample violating the Condition (F) when $(X_0, X)$ is not maximal.

\section*{Acknowledgement}
I would like to thank Prof. Ngaiming Mok for his insightful guidance. I am also grateful to Prof. Baohua Fu and Feng Shao for some helpful discussions.

\section{Preliminaries}\label{preli}
We collect some notions and facts from \cite{Zhang2014}, \cite{MR3960268}, \cite{MR4028278}.
We first give the definition of admissible pairs.

\subsection{Admissible pairs of deletion type}
\begin{definition}[admissible pairs]
Let $X_0$ and $X$ be rational homogeneous space of Picard number one, and $i: X_0\hookrightarrow X$ be a holomorphic embedding equivariant with respect to a homomorphism of complex Lie groups $\Phi: \Aut_0(X_0)\rightarrow \Aut_0(X)$. We say that $(X_0,X;i)$ (sometimes $i$ is omitted) is an admissible pair if and only if (a) $i$ induces an isomorphism $i_*: H_2(X_0, \mathbb{Z})\xrightarrow{\cong} H_2(X, \mathbb{Z})$, and (b) denoting by $\mathcal{O}(1)$ the positive generator of $\Pic(X)$ and by $\rho: X\hookrightarrow \mathbb{P}(\Gamma(X, \mathcal{O}(1))^*)=:\mathbb{P}^N$ the first canonical projective embedding of $X$, $\rho \circ i: X_0\hookrightarrow \mathbb{P}^N$ embeds $X_0$ as a smooth linear section of $\rho(X)$.
\end{definition}

In \cite{Zhang2014}, admissible pairs of Hermitian type are classified as subdiagram type, deletion type and special type, where the deletion type is defined as follows.
\begin{definition}[admissible pairs of deletion type]\label{deletion}
	Let $(X_0, X)$ be an admissible pair of irreducible Hermitian symmetric spaces of compact type
	associated to the marked Dynkin diagrams $(\mathcal{D}_0, \gamma_0)$ and $(\mathcal{D}, \gamma)$ respectively, then the pair is called deletion type if we can obtain $X_0$ by deleting a chain connecting $\gamma$ and $\gamma_0$, where a chain means a type A Dynkin diagram.
\end{definition}	

The full classification can be given as follows.

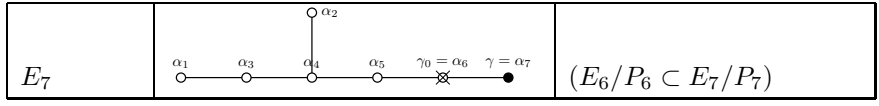
\begin{center}
	\begin{longtable}{|p{4.5em}|p{14em}|p{11em}|}
		\caption{Admissible pairs of deletion type}\label{Admissible pairs of deletion type}\\
		\hline
		Type of $X$  & Marked Dynkin diagram   & $(X_0\subset X)$ 
		\\ \hline
		$B_n$ &  \resizebox{\linewidth}{!}{%
			\begin{tikzpicture}
			\draw[thick] (0,0) -- (3,0) ;
			\draw[thick, dashed] (3,0) -- (6,0) ;
			\draw[->-,thick, double] (6,0) -- (7.5,0) ;
			\draw[ thick, fill=black] (0,0) circle (3pt) node[above, outer sep=3pt]{$\gamma=\alpha_1$};
			\draw[ thick, fill=white] (1.5,0) circle (3pt) node[above, outer sep=3pt]{$\alpha_2$};
			\draw[ thick, fill=white] (3,0) circle (3pt) node[above, outer sep=3pt]{$\alpha_3$};
			\draw[ thick, fill=white] (4.5,0) circle (3pt) node[above,outer sep=3pt]{$\gamma_0=\alpha_m$};
			\draw[ thick, fill=white] (6,0) circle (3pt) node[above,outer sep=3pt]{$\alpha_{n-1}$};
			\draw[ thick, fill=white] (7.5,0) circle (3pt) node[above,outer sep=3pt]{$\alpha_n$};
			\draw (4.5, 0) node[cross] {};
			\end{tikzpicture}}
		& $(Q^{2n-2m+1}\subset Q^{2n-1})$ \\
		\hline 
		
		$D_n$ & \resizebox{\linewidth}{!}{%
			\begin{tikzpicture}
			\draw[thick] (0,0) -- (3,0) (4.5,0)--(6,0) (6,0)--(7.2,0.4) (6,0)--(7.2,-0.4);
			\draw[thick, dashed] (3,0) -- (4.5,0);
			
			\draw[thick, fill=white] (0,0) circle (3pt) node[above, outer sep=3pt]{$\alpha_1$};
			\draw[ thick, fill=white] (1.5,0) circle (3pt) node[above, outer sep=3pt]{$\alpha_2$};
			\draw[ thick, fill=white] (3,0) circle (3pt) node[above, outer sep=3pt]{$\alpha_3$};
			\draw[ thick, fill=white] (4.5,0) circle (3pt) node[above,outer sep=3pt]{$\alpha_{n-3}$};
			\draw[ thick, fill=white] (6,0) circle (3pt) node[above,outer sep=3pt]{$\gamma_0=\alpha_{n-2}$};
			\draw[ thick, fill=white] (7.2,0.4) circle (3pt) node[above,outer sep=3pt]{$\alpha_{n-1}$};
			\draw[ thick, fill=black] (7.2,-0.4) circle (3pt) node[below,outer sep=3pt]{$\gamma=\alpha_{n}$};
			\draw (6, 0) node[cross] {};
			\end{tikzpicture}}
		&$(G(2,n-2)\subset G^{II}(n,n))$\\
		\hline
		
		$D_n$ & \resizebox{\linewidth}{!}{%
			\begin{tikzpicture}
			\draw[thick] (0,0) -- (1.5,0) (4.5,0)--(6,0) (6,0)--(7.2,0.4) (6,0)--(7.2,-0.4);
			\draw[thick, dashed] (1.5,0) -- (4.5,0);
			
			\draw[thick, fill=black] (0,0) circle (3pt) node[above, outer sep=3pt]{$\gamma=\alpha_1$};
			\draw[ thick, fill=white] (1.5,0) circle (3pt) node[above, outer sep=3pt]{$\alpha_2$};
			\draw[ thick, fill=white] (3,0) circle (3pt) node[above, outer sep=3pt]{$\gamma_0=\alpha_m$};
			\draw[ thick, fill=white] (4.5,0) circle (3pt) node[above,outer sep=3pt]{$\alpha_{n-3}$};
			\draw[ thick, fill=white] (6,0) circle (3pt) node[above,outer sep=3pt]{$\alpha_{n-2}$};
			\draw[ thick, fill=white] (7.2,0.4) circle (3pt) node[above,outer sep=3pt]{$\alpha_{n-1}$};
			\draw[ thick, fill=white] (7.2,-0.4) circle (3pt) node[above,outer sep=3pt]{$\alpha_{n}$};
			\draw (3, 0) node[cross] {};
			\end{tikzpicture}}
		&$(Q^{2n-2m}\subset Q^{2n-2})$\\
		\hline
		
		$E_6$ & 
		\resizebox{\linewidth}{!}{%
			\begin{tikzpicture}
			\draw[thick] (0,0) -- (6,0)  (3,0) -- (3,1.5) ;
			\draw[ thick, fill=white] (0,0) circle (3pt) node[above, outer sep=3pt]{$\alpha_1$};
			\draw[ thick, fill=white] (1.5,0) circle (3pt) node[above, outer sep=3pt]{$\alpha_3$};
			\draw[ thick, fill=white] (3,0) circle (3pt) node[above, outer sep=3pt]{$\gamma_0=\alpha_4$};
			\draw[ thick, fill=white] (4.5,0) circle (3pt) node[above,outer sep=3pt]{$\alpha_5$};
			\draw[ thick, fill=black] (6,0) circle (3pt) node[above,outer sep=3pt]{$\gamma=\alpha_6$};
			\draw[ thick, fill=white] (3,1.5) circle (3pt) node[right,outer sep=3pt]{$\alpha_2$};
			\draw (3, 0) node[cross] {};
			\end{tikzpicture}}
		& $(G(2,3) \subset E_6/P_6)$ \\
		\hline
		
		$E_6$ & 
		\resizebox{\linewidth}{!}{%
			\begin{tikzpicture}
			\draw[thick] (0,0) -- (6,0)  (3,0) -- (3,1.5) ;
			\draw[ thick, fill=white] (0,0) circle (3pt) node[above, outer sep=3pt]{$\alpha_1$};
			\draw[ thick, fill=white] (1.5,0) circle (3pt) node[above, outer sep=3pt]{$\alpha_3$};
			\draw[ thick, fill=white] (3,0) circle (3pt) node[above, outer sep=3pt]{$\alpha_4$};
			\draw[ thick, fill=white] (4.5,0) circle (3pt) node[above,outer sep=3pt]{$\gamma_0=\alpha_5$};
			\draw[ thick, fill=black] (6,0) circle (3pt) node[above,outer sep=3pt]{$\gamma=\alpha_6$};
			\draw[ thick, fill=white] (3,1.5) circle (3pt) node[right,outer sep=3pt]{$\alpha_2$};
			\draw (4.5, 0) node[cross] {};
			\end{tikzpicture}}
		& $(G^{II}(5,5) \subset E_6/P_6)$ \\
		\hline
		$E_7$ &  
		\resizebox{\linewidth}{!}{%
			\begin{tikzpicture}
			\draw[thick] (0,0) -- (7.5,0)  (3,0) -- (3,1.5) ;
			\draw[ thick, fill=white] (0,0) circle (3pt) node[above, outer sep=3pt]{$\alpha_1$};
			\draw[ thick, fill=white] (1.5,0) circle (3pt) node[above, outer sep=3pt]{$\alpha_3$};
			\draw[ thick, fill=white] (3,0) circle (3pt) node[above, outer sep=3pt]{$\gamma_0=\alpha_4$};
			\draw[ thick, fill=white] (4.5,0) circle (3pt) node[above,outer sep=3pt]{$\alpha_5$};
			\draw[ thick, fill=white] (6,0) circle (3pt) node[above,outer sep=3pt]{$\alpha_6$};
			\draw[ thick, fill=black] (7.5,0) circle (3pt) node[above,outer sep=3pt]{$\gamma=\alpha_7$};
			\draw[ thick, fill=white] (3,1.5) circle (3pt) node[right,outer sep=3pt]{$\alpha_2$};
			\draw (3, 0) node[cross] {};
			\end{tikzpicture}	}
		& $(G(2,3) \subset E_7/P_7)$ \\
		\hline
		
		$E_7$ &  
		\resizebox{\linewidth}{!}{%
			\begin{tikzpicture}
			\draw[thick] (0,0) -- (7.5,0)  (3,0) -- (3,1.5) ;
			\draw[ thick, fill=white] (0,0) circle (3pt) node[above, outer sep=3pt]{$\alpha_1$};
			\draw[ thick, fill=white] (1.5,0) circle (3pt) node[above, outer sep=3pt]{$\alpha_3$};
			\draw[ thick, fill=white] (3,0) circle (3pt) node[above, outer sep=3pt]{$\alpha_4$};
			\draw[ thick, fill=white] (4.5,0) circle (3pt) node[above,outer sep=3pt]{$\gamma_0=\alpha_5$};
			\draw[ thick, fill=white] (6,0) circle (3pt) node[above,outer sep=3pt]{$\alpha_6$};
			\draw[ thick, fill=black] (7.5,0) circle (3pt) node[above,outer sep=3pt]{$\gamma=\alpha_7$};
			\draw[ thick, fill=white] (3,1.5) circle (3pt) node[right,outer sep=3pt]{$\alpha_2$};
			\draw (4.5, 0) node[cross] {};
			\end{tikzpicture}	}
		& $(G^{II}(5,5) \subset E_7/P_7)$ \\
		\hline
		
		$E_7$ &  
		\resizebox{\linewidth}{!}{%
			\begin{tikzpicture}
			\draw[thick] (0,0) -- (7.5,0)  (3,0) -- (3,1.5) ;
			\draw[ thick, fill=white] (0,0) circle (3pt) node[above, outer sep=3pt]{$\alpha_1$};
			\draw[ thick, fill=white] (1.5,0) circle (3pt) node[above, outer sep=3pt]{$\alpha_3$};
			\draw[ thick, fill=white] (3,0) circle (3pt) node[above, outer sep=3pt]{$\alpha_4$};
			\draw[ thick, fill=white] (4.5,0) circle (3pt) node[above,outer sep=3pt]{$\alpha_5$};
			\draw[ thick, fill=white] (6,0) circle (3pt) node[above,outer sep=3pt]{$\gamma_0=\alpha_6$};
			\draw[ thick, fill=black] (7.5,0) circle (3pt) node[above,outer sep=3pt]{$\gamma=\alpha_7$};
			\draw[ thick, fill=white] (3,1.5) circle (3pt) node[right,outer sep=3pt]{$\alpha_2$};
			\draw (6, 0) node[cross] {};
			\end{tikzpicture}	}
		& $(E_6/P_6 \subset E_7/P_7)$ \\
		\hline
		\end{longtable}
\end{center}

From \cite{Zhang2014} we know for such admissible pairs of Hermitian type, the embedding is induced from a root correspondence.
Let $\Gamma$ be the sum of simple roots in the chain connecting $\gamma$ and $\gamma_0$ (containing $\gamma_0$ but not $\gamma$) in $\mathcal{D}$.
The root correspondence $\Phi$ for deletion type is given by the following \[ \Phi(\beta)=
\begin{cases}
\beta, \beta \ \text{is neither $\gamma_0$ nor adjacent to $\gamma_0$}\\
\gamma, \beta=\gamma_0\\
\beta+\Gamma, \beta \ \text{is adjacent to $\gamma_0$}
\end{cases}
\]

\subsection{Geometric substructure and the rigidity}
The geometric substructure is given by
\begin{definition}[sub-VMRT structure]\label{subvmrt}
	Let $(X_0, X)$ be an admissible pair of rational homogeneous spaces of Picard number one, $W\subset X$ be a connected open subset in the complex topology, and $S\subset W$ be a complex submanifold. Consider the fibered space $\pi: \mathcal{C}(X)\rightarrow X$ of varieties of minimal rational tangents on $X$. For every point $x\in S$, define $\mathcal{C}_x(S):=\mathcal{C}_x(X)\cap \mathbb{P}T_x(S)$ and write $\overline{\omega}=\pi|_{\mathcal{C}(S)}$, $\overline{\omega}^{-1}(x)=\mathcal{C}_x(S)$ for $x\in S$. We say $S\subset W$ inherits a sub-VMRT structure modeled on $(X_0, X)$ if and only if for every point $x \in S$ there exists a neighborhood $U$ of $x$ on $S$ and a trivialization of the holomorphic projective bundle $\mathbb{P}T_{X}|_{U}$ given by $\Phi: \mathbb{P}T_X|_U \xrightarrow{\cong} U\times \mathbb{P}T_0(X)$ such that (1) $\Phi(\mathcal{C}(X)|_U)=U\times \mathcal{C}_0(X)$ and (2) $\Phi(\mathcal{C}(S)|_U)=U\times \mathcal{C}_0(X_0)$
\end{definition}

Then the rigidity of admissible pairs can be defined as follows.

\begin{definition}[rigidity of admissible pairs]\label{rigidity}
	An admissible pair $(X_0, X)$ of rational homogeneous spaces of Picard number one is said to be rigid if and only if for any connected open subset $W\subset X$, any submanifold $S\subset W$ inheriting a sub-VMRT structure modeled on $(X_0, X)$ must necessarily be an open subset of $\gamma(X_0)\subset X$ for some $\gamma\in \Aut(X)$.
\end{definition}

From \cite{MR3960268} the rigidity can be reduced to check a nondegeneracy condition in terms of the projective second fundamental forms. To be more precise, we give

\begin{definition}[non-degeneracy for substructures]\label{nondegeneracy}
Let $(X_0, X)$ be an admissible pair of rational homogeneous spaces of Picard number one. Let $0\in X_0\subset X$ be a reference point and $\alpha \in \widetilde{\mathcal{C}_0(X_0)}$ (the affinized VMRT) be arbitrary. Write $P_\alpha=T_\alpha(\widetilde{\mathcal{C}_0(X)})$, and denote by $\sigma_{\alpha}:S^2P_\alpha\rightarrow T_0(X)/P_\alpha$ the second fundamental form of $\widetilde{\mathcal{C}_0(X)}\subset T_0(X)-\{0\}$ at $\alpha$ with respect to the flat connection on $T_0(X)$ as a Euclidean space. Denote by $\nu_\alpha: T_0(X)/P_\alpha\rightarrow T_0(X)/(P_\alpha+D_0\cap T_0(X_0))$ the canonical projection, where $D_0=\Span(\widetilde{\mathcal{C}_0(X)})$. Write $\tau_\alpha:=\nu_\alpha\circ \sigma_\alpha$, so that $\tau_\alpha: S^2P_\alpha\rightarrow T_0(X)/(P_\alpha+D_0\cap T_0(X_0))$, define
\[Ker \tau_{\alpha}(\cdot,T_\alpha(\widetilde{\mathcal{C}_0(X_0)})):=\{\eta\in T_\alpha(\widetilde{\mathcal{C}_0(X)}): \tau_{\alpha}(\eta,\xi)=0 \ \forall \xi\in T_\alpha(\widetilde{\mathcal{C}_0(X_0)}) \}\]
We say that $(\mathcal{C}_0(X_0), \mathcal{C}_0(X))$ is nondegenerate for substructures if and only if for a general point $\alpha\in \widetilde{\mathcal{C}_0(X_0)}$ we have $Ker \tau_{\alpha}(\cdot,T_\alpha(\widetilde{\mathcal{C}_0(X_0)}))=T_\alpha(\widetilde{\mathcal{C}_0(X_0)})$.
\end{definition}

In \cite{MR3960268}, they obtained that 

\begin{proposition}
	If $(\mathcal{C}_0(X_0),\mathcal{C}_0(X))$ is non-degenerate for substructure, then for any connected open subset $W\subset X$, any submanifold $S\subset W$ inheriting a sub-VMRT structure modeled on $(X_0, X)$ must be rational saturated.
\end{proposition}
 In \cite{Zhang2014}, degeneracy for substructure appears for admissible pairs of deletion type. This is the reason why we need the assumption of rational saturation in Theorem \ref{thmdeletion}. Since the parallel transport argument in \cite{MR3960268} fails for admissible pairs of deletion type, Mok \cite{MR4028278} introduced a new criterion to check the rigidity, as stated in Theorem \ref{Mokthmsaturation}.
 
 Finally, since in \cite{Zhang2014}, the degeneracy for substructures in deletion type cases is mentioned without proof, we give a proof here.
 \begin{proposition}
 	Let $X_0, X$ be irreducible Hermitian symmetric spaces of compact type.
 	If $(X_0, X)$ is an admissible pair of deletion type, then $(\mathcal{C}_0(X_0), \mathcal{C}_0(X))$ is degenerate for substructures.
 \end{proposition}
\begin{proof}
From \cite[Lemma 3.1]{MR3960268}, it suffices to show that $(\mathcal{C}_0(X_0), \mathcal{C}_0(X))$ is degenerate for mappings, i.e. for a general point $\alpha\in \widetilde{\mathcal{C}_0(X_0)}$ we have \[Ker \sigma_{\alpha}(\cdot,T_\alpha(\widetilde{\mathcal{C}_0(X_0)})):=\{\eta\in T_\alpha(\widetilde{\mathcal{C}_0(X)}): \sigma_{\alpha}(\eta,\xi)=0 \ \forall \xi\in T_\alpha(\widetilde{\mathcal{C}_0(X_0)}) \}=\mathbb{C}\alpha.\]
 When $X$ is a hyperquadric, the conclusion is easy to see as $\sigma$ is precisely the holomorphic conformal structure of a hyperquadric. Then for our convenience assume that $X$ is not of $B_n$ type and all roots are of equal length.
Without loss of generality we may assume the point $\alpha=E_{\gamma}$ where $\gamma$ is the marked root.
From \cite[Lemma 3.2]{MR3960268},  the second fundamental form $\sigma_\alpha$ can be written as 
\[\sigma_\alpha(E_{\nu}, E_{\nu'})=[E_{\nu-\gamma},[E_{\nu'-\gamma}, E_{\gamma}]]\mod (P_\alpha+\mathfrak{p})\]
where $\nu, \nu'\in \Psi_\gamma:=\{\mu, \mu-\gamma \ \text{is a root or}\ 0  \}$ and $E_{\nu}, E_{\nu'}\in T_\alpha(\widetilde{\mathcal{C}_0(X)})=\Span\{E_\mu, \mu\in \Psi_\gamma\}$.  Similarly we write $\Psi_{\gamma_0}:=\{\mu, \mu-\gamma_0 \ \text{is a root with root space in } \mathfrak{g}_0 \ \text{or} \ 0 \}$.  Then after taking the root correspondence we have $T_\alpha(\widetilde{\mathcal{C}_0(X_0)})=\mathbb{C}E_{\gamma}+\Span\{E_\mu, \mu=\gamma+\Gamma+\kappa, \kappa\neq 0,\kappa+ \gamma_0 \in \Psi_{\gamma_0} \}$. Take a root vector $E_{\gamma+\beta}$ where $\beta$ is adjacent to $\gamma$, since $<\beta, \gamma+\Gamma+\kappa>=-1+2-1=0$ and all roots are of equal length, we know $\beta+\gamma+\Gamma+\kappa$ is not a root. Hence $\sigma_\alpha(E_{\gamma+\beta}, T_\alpha(\widetilde{\mathcal{C}_0(X_0)}))=0$ and then $(\mathcal{C}_0(X_0), \mathcal{C}_0(X))$ is degenerate for substructures..
\end{proof}

\subsection{Formulation of Condition (F)}
We give the Condition (F) in Theorem \ref{Mokthmsaturation} as follows.
Suppose that $(X_0, X)$ is an admissible pair of rational homogeneous spaces of Picard number one, we know there is a corresponding natural embedding of their VMRTs $i:\mathcal{C}_0(X_0) \hookrightarrow \mathcal{C}(X)$, we call an embedding $i':\mathcal{C}_0(X_0) \hookrightarrow \mathcal{C}(X)$ is standard if $i'=g|_{\mathcal{C}(X)}\circ i$ where $g$ is a projective automorphism of $\mathbb{P}T_0(X)$ preserving $\mathcal{C}_0(X)$. 

Let $\lambda: \mathcal{C}_0(X_0) \hookrightarrow \mathcal{C}_0(X)$ be an arbitrary holomorphic embedding satisfying (i)$\mathbb{P}(\Span\widetilde{\im(\lambda)})\cap \mathcal{C}_0(X)=\lambda(\mathcal{C}_0(X_0))$ and (ii)
$\im(\lambda)\subset \mathbb{P}(\Span\widetilde{\im(\lambda)})$ is projectively equivalent to the inclusion $(\mathcal{C}_0(X_0)\subset \mathbb{P}T_0(X_0))$
, we say the pair $(\mathcal{C}_0(X_0), \mathcal{C}(X))$ satisfies Condition (F) if $\lambda$ is a standard embedding.

\section{Sketch of Mok's proof of Theorem \ref{Mokthmsaturation}}\label{sketch}
In this section we give a sketch of Mok's proof of Theorem \ref{Mokthmsaturation}. The key ingredient is the following proposition on algebracity of the germ, which tells that the germ inheriting a sub-VMRT structure modeled on $(X_0, X)$ is contained in some projective subvariety whose normalization is biholomorphic to $X_0$, under the assumption of rational saturation. We start with

\begin{proposition}[Proposition 3.1 in \cite{MR4028278}]
	Let $(X_0, X)$ be an admissible pair of irreducible Hermitian symmetric spaces of compact type of rank $\geq 2$, $0\in X_0\subset X$. Suppose that $(\mathcal{C}_0(X_0), \mathcal{C}_0(X))$ satisfies Condition (F). Let $W\subset X$ be an open subset in complex topology, and $S\subset W$ be a complex submanifold inheriting a sub-VMRT structure modeled on $(X_0, X)$ and $S$ is rational saturated. Then there exists a projective subvariety $Z\subset X$ such that $S\subset Z$ and $\dim(Z)=\dim(S)$. Moreover writing $\nu: \widetilde{Z}\rightarrow Z$ for the normalization, there is a biholomorphism $\Phi: X_0 \xrightarrow{\cong} \widetilde{Z}$, and for any choice of such a biholomorphism $\Phi$ the holomorphic map $h:=\mu\circ \Phi :X_0\rightarrow X$ is of degree 1, i.e. $h^*(\mathcal{O}(1))\cong \mathcal{O}(1)$.
\end{proposition} 
\begin{proof}[Sketch of the proof]
	For any point $x\in S$, take any minimal rational curve $\ell$ on $X$ emanating from $x$ such that $[T_x(\ell)]=:[\alpha]\in \mathcal{C}_x(S)=\mathcal{C}_x(X)\cap \mathbb{P}T_x(S)$, so that $(\ell ;x)\subset (S; x)$ by linear saturation. 
	The first step is to construct a collar $\mathbf{N}_\ell$ around $\ell$ and propagate the fibered space $\overline{\omega}: \mathcal{C}(S)\rightarrow S$ to $\mathbf{N}_\ell$ such that it is still isotrivial with fibers isomorphic to $\mathcal{C}_0(X_0)\subset \mathbb{P}T_0(X_0)$, in other words, we want to propagate the $L_0$-structure on $S$ to $\mathbf{N}_\ell$ where $L_0$ is a reductive part of $P_0$.

 $\mathcal{C}_x(S)$ is propagated by parallel transport from $x\in S$ to $y\in \textbf{N}_\ell\backslash S$ along minimal rational curves lying in the collar, i.e. choose a minimal rational curve $\Lambda$ of $X$ lying in $\textbf{N}_\ell$, $(\mathcal{C}_y(\textbf{N}_\ell)\subset \mathbb{P}T_y(\textbf{N}_\ell))$ is projectively equivalent to $(\mathcal{C}_x(\textbf{N}_\ell)\subset \mathbb{P}T_x(\textbf{N}_\ell))$ for \textit{arbitrary} point $y\in \Lambda$. 
 The propagation is guaranteed by using the lifting of minimal rational curves and the relative second fundamental form.  The main difficulty is that when we do the parallel transport, we can only propagate $(\mathcal{C}_x(\textbf{N}_\ell)\subset \mathbb{P}T_x(\textbf{N}_\ell))$, but not the pair $(\mathcal{C}_x(\textbf{N}_\ell), \mathcal{C}_x(X))$ to a point $y$ outside $S$. Thus we need the \textit{fitting} of sub-VMRTs into VMRTs and Condition (F) is required.
 
	Then we remain to use the adjunction by minimal rational curves and do the iteration. Finally we obtain an immersed submanifold $Z$ whose normalization $\widetilde{Z}\rightarrow Z$ is nonsingular and projective. By the recognition theorem given by Hwang and Mok \cite[Theorem 2.2]{MR4028278}, $\widetilde{Z}$ is biholomorphic to $X_0$.
\end{proof}
\begin{remark}
	In fact the Condition (F) is automatically satisfied on general points. By Chevalley's theorem, any $P$-orbit in $Gr(\dim(X_0),T_x(X))$ is constructible and moreover locally closed (open and dense in its Zariski closure).  Hence, to check the Condition (F), it is equivalent to show that Chow point in the boundary of $P.[\mathcal{C}_0(X_0)]$ is corresponding to a singular subvariety in $\mathcal{C}_0(X)$.
\end{remark}
Now we give a sketch of of proof of Theorem \ref{Mokthmsaturation}.

\begin{proof}[Sketch of proof of Theorem \ref{Mokthmsaturation}]
Define $Z_t:=T_{\frac{1}{t}}(Z)=\{w:tw\in Z\}$. The first step is to show that $Z_t$ converges to $X_0$ as $t\rightarrow 0$. This can be deduced from 
\[\begin{aligned}\text{Volume}(\text{Reg}(Z_t), \omega|_{\text{Reg}(Z_t)})&= \text{Volume}(\text{Reg}(Z),  \omega|_{\text{Reg}(Z)})  \\&=\text{Volume}(X_0,  h^*\omega)=\text{Volume}(X_0, \omega|_{X_0}) \end{aligned}\]
 where $\omega$ is the K\"{a}hler form such that minimal rational curves are of unit volume, $h: X_0\rightarrow X$ is a holomorphic immersion (a rational map) with $h(X_0)=Z$. The first equality holds since $t$ is from a $\mathbb{C}^*$-action; the second equality is from the fact that $X_0$ is biholomorphic to the normalization of $Z$; the third equality holds since $h^*(\mathcal{O}(1))=\mathcal{O}(1)$. Since the difference between the limit of $Z_t$ and $X_0$ is a cycle supported in the boundary divisor of $X$ (as complement of an affine cell), the volume equality forces that $Z_t\rightarrow X_0$ as $t\rightarrow 0$.
 
 Next step is to show that $h$ is actually a holomorphic embedding. This is deduced from the semicontinuity of Lelong numbers, i.e. 
 \[\nu([X_0],q)\geq \limsup_{n\rightarrow \infty}\nu([Z_{\frac{1}{n}}],p_n)=\nu([Z],p)\geq 2\]
 where $p\in \text{Sing}(Z), p_n=T_n(Z)$ and $q$ is a limit point of $\{p_n\}_{n\geq 1}$, $\nu(S, x)$ deontes the Lelong number of a closed positive current $S$ on $X$ at $x$. Contradiction arises since $X_0$ is smooth. Now it remains to show that when $t$ is close to $0$, $Z_{t}=\gamma_t(X_0)$ for some $\gamma_t\in \Aut_0(X)$, i.e. local deformation of $X_0\subset X$ must be induced by $\Aut_0(X)$. This will be guaranteed by the surjectivity of the restriction map $r: \Gamma(X, T_X)\rightarrow \Gamma(X_0, T_{X|X_0})$ (see \cite[Lemma 5.1]{MR4028278}).
\end{proof}

\section{Checking the Condition (F)}\label{checkingconditionF}
In this section, we check the Condition (F) when $(X_0, X)$ is of deletion type and maximal. For our convenience we generalize the definition of admissible pairs such that $X_0$ is allowed to be of higher Picard number.
\begin{definition}\label{admissible}
	Let $X$ be a rational homogeneous space of Picard number one, and $X_0$ be a rational homogeneous space of Picard number $r$, and $i:X_0 \hookrightarrow X$ be a holomorphic embedding equivariant with respect to a homomorphism of complex Lie groups $\Phi :\Aut_0(X_0) \rightarrow \Aut_0(X)$. We say that $(X_0, X;i)$ is an admissible pair if and only if (a) $i$ induces a homomorphsim $i_*$ : $H_2(X_0, \mathbb{Z})\rightarrow H_2(X,\mathbb{Z})$ which maps the each generator of $H_2(X_0, \mathbb{Z})$ into the generator of $H_2(X,\mathbb{Z})$, and (b) denoting by $\mathcal{O}(1)$ the positive generator of $\Pic(X)$ and by $\rho: X \hookrightarrow \mathbb{P}(\Gamma(X, \mathcal{O}(1))^*)=:\mathbb{P}^N$ the first canonical projective embedding of $X$, $\rho\circ i: X_0 \hookrightarrow \mathbb{P}^N$ embeds $X_0$ as a smooth linear section of $\rho(X)$.
	
\end{definition}

We first prove

\begin{lemma}\label{firstfitting}
	If $\{pt\}\cup \mathbb{P}^1$ is embedded in $\mathbb{P}^1\times \mathbb{P}^2\hookrightarrow \mathbb{P}^5$ as a linear section by $\mathbb{P}^2$ and $\mathbb{P}^1$ is embedded in $ \mathbb{P}^1\times \mathbb{P}^2$ of bidegree $(1,0)$ or $(0,1)$. Then the images of all such embeddings are transitive by $\Aut(\mathbb{P}^1\times \mathbb{P}^2).$
\end{lemma}

\begin{proof}
If $\mathbb{P}^1$ is embedded as $\mathbb{P}^1\times \{y\}\subset \mathbb{P}^1\times \mathbb{P}^2$ for some $y\in \mathbb{P}^2$ in the second factor and the remaining point is written as $(a,b)\in \mathbb{P}^1\times \mathbb{P}^2$, we join $(a,y)$ with $(a,b)$ and then obtain a rational curve in the $\mathbb{P}^1\times \mathbb{P}^2$ of bidegree $(0,1)$. The curve is in $ \mathbb{P}^2\cap (\mathbb{P}^1\times \mathbb{P}^2)$ distinct from $\mathbb{P}^1\times \{y\}$, a contradiction. Hence $\mathbb{P}^1$ must be embedded as $\{x\}\times \mathbb{P}^1\subset \mathbb{P}^1\times \mathbb{P}^2$ for some $x\in \mathbb{P}^1$ in the first factor. Now we also write the remaining point as $(a,b)$. If $a=x$, $ \mathbb{P}^2\cap (\mathbb{P}^1\times \mathbb{P}^2)$ contains a $\mathbb{P}^2$, also a contradiction. Hence $a\neq x$. 
For similar reason, $b$ is not in the line $\mathbb{P}^1\subset \mathbb{P}^2$. Then it remains to prove that if we fix $(a, b)\in \mathbb{P}^1\times \mathbb{P}^2$, $P_{(a,b)}:=Stab((a,b))$ acts transitively on space of lines  $\{x\}\times \mathbb{P}^1\subset \mathbb{P}^1\times \mathbb{P}^2$ with $x\neq a, b\notin \mathbb{P}^1$. The checking is simple.
\end{proof}

The lemma actually gives the 'Fitting' of the embedding of $\{pt\}\cap \mathbb{P}^1$ into $\mathbb{P}^1\times \mathbb{P}^2$.
Admissible pairs of deletion type and maximal can be exhausted by taking the VMRTS of $(E_6/P_6, E_7/P_7)$, more precisely, they are $(E_6/P_6, E_7/P_7)\xrightarrow{\text{take VMRTS}} (G^{II}(5,5), E_6/P_6)\xrightarrow{\text{take VMRTS}} (G(2,3), G^{II}(5,5))$. If we continue taking the VMRTs, we have $(G(2,3), G^{II}(5,5))\xrightarrow{\text{take VMRTS}} (\mathbb{P}^1\times \mathbb{P}^2, G(2,3))$. Then together with Lemma \ref{firstfitting}
it suffices to prove the following proposition.

\begin{proposition}
	Let $\lambda: X_0 \rightarrow X\hookrightarrow \mathbb{P}^N$ be an arbitraty holomorphic embedding satisfies (i) $\mathbb{P}(\Span\widetilde{\im(\lambda)})\cap X=\lambda(X_0)$ and (ii)
	$\im(\lambda)\subset \mathbb{P}(\Span\widetilde{\im(\lambda)})$ is projectively equivalent to the inclusion $(X_0, X)$. If $(\mathcal{C}_0(X_0), \mathcal{C}(X))$ satisfies Condition (F), then $\lambda$ is a standard embedding, i.e. under $\Aut(X)$ it is equivalent to the embedding of $X_0$ in $X$ as an admissible pair of deletion type. 
\end{proposition}

\section{Checking the surjectivity of the restriction map}\label{checkingsujectivity}
We first prove the lemma
\begin{lemma}\label{deletionasvmrt}
	Suppose that $X$ is not of $A_n$ (the VMRT is not of Picard number one) or $C_n$ type (the VMRT is a Veronese embedding of a projective space).
	Let $\mathcal{V}_x$ deonte the cone swept out by minimal rational curves in $X$ marked at $x$. Let $\mathcal{W}\subset X$ be a Harish-Chandra coordinate such that $x$ is the zero point.  Let \[\begin{aligned}
	\mathcal{N}=&\{C\backslash \mathcal{W}: C \ \text{is a minimal rational curve passing through} \ x \} 
	\end{aligned}\] be the locus of infinity points of those minimal rational curves. Then the embedding $\mathcal{N}\subset X$ gives the admissible pair of deletion type and it is maximal.
\end{lemma}
\begin{proof}
	From \cite[Proposition  3.4]{ding2022birational} we know after suitable group action $g$ we have $x\in g.\mathcal{N}$ and $T_x(g.\mathcal{N})=\mathbb{C}\{E_{\beta}\}_{\beta\in \Delta^+_{\mathfrak{n}}, <\beta, \gamma>=1}$ where $ \Delta^+_{\mathfrak{n}}$ is the set of positive noncompact roots. On the other hand, let\[ \begin{aligned}
	&
	\mathcal{H}_{\gamma_0}=\mathbb{C}\{E_{\beta}\}_{\beta\in \Delta^+_{\mathfrak{n},0}, <\beta, \gamma_0>=1 }  \\&
	 \mathcal{N}_{\gamma_0}=\mathbb{C}\{E_{\beta}\}_{\beta\in \Delta^+_{\mathfrak{n},0}, <\beta, \gamma_0>=0 } \end{aligned}\]
	 where $\Delta^+_{\mathfrak{n}, 0}$ is the set of positive noncompact roots for $\mathfrak{g}_0$. From  \cite[Lemma 1.11]{Zhang2014}, we know for any $\beta\in \Delta^+_{\mathfrak{n},0}$ with $<\beta, \gamma_0>=1$, we can write $\beta$ as $\beta=\gamma_0+\theta+\Sigma$ where $\theta$ is the root adjacent to $\gamma_0$ in $\mathcal{D}(G_0)$ and $\Sigma$ is a linear combination of roots which are not adjacent to $\gamma_0$ in  $\mathcal{D}(G_0)$.
	 After the root correspondence in the admissible pair of deletion type, such $\beta$ is corresponding to the root $\gamma+\gamma_0+\theta+\Sigma$ and we know $<\gamma, \gamma+\gamma_0+\theta+\Sigma>=1$.
	 Then after the action by the simple reflection $s_{\gamma_0}$ in the Weyl group, we know $s_{\gamma_0}(\gamma+\gamma_0+\theta+\Sigma)=\gamma+\gamma_0+\theta+\Sigma$.
	 
	 Also from  \cite[Lemma 1.11]{Zhang2014}, we know for any $\beta\in \Delta^+_{\mathfrak{n},0}$ with $<\beta, \gamma_0>=0$, we can write $\beta$ as $\beta=\gamma_0+2\theta+\Sigma$ where $\theta$ is the root adjacent to $\gamma_0$ in $\mathcal{D}(G_0)$ and $\Sigma$ is a linear combination of roots which are not adjacent to $\gamma_0$ in  $\mathcal{D}(G_0)$.  After the root correspondence in the admissible pair of deletion type, such $\beta$ is corresponding to the root $\gamma+2\gamma_0+2\theta+\Sigma$ and we know $<\gamma, \gamma+2\gamma_0+2\theta+\Sigma>=0$. Then after the action by the simple reflection $s_{\gamma_0}$ in the Weyl group, we know $s_{\gamma_0}(\gamma+2\gamma_0+2\theta+\Sigma)=\gamma+\gamma_0+2\theta+\Sigma$.
	 
	 Since $<\gamma, \gamma+\gamma_0+2\theta+\Sigma>=1$, $s_{\gamma_0}(\gamma)=\gamma+\gamma_0$ and $<\gamma, \gamma+\gamma_0>=1$, in summary we have 
	$s_{\gamma_0}.T_x(X_0)\subset T_x(g\mathcal{N})$ and hence they are equal. Then the lemma follows from \cite[Proposition 2.4]{Zhang2014}.
	
\end{proof}

Then we have
\begin{proposition}\label{decomp}
	When $(X_0, X)$ is of deletion type and maximal,
	the normal bundle $N_{X_0|X}$ is a direct sum of $\mathcal{O}_{X_0}(1)$ and an irreducible $P_0$-homogeneous bundle.
\end{proposition}	
\begin{proof}
 We consider the tautological  $\mathbb{P}^1$-bundle over $\mathcal{C}(X)$ denoted by $\mathcal{M}$. Let $ev: \mathcal{M}\rightarrow X$ be the evaluation map. Consider two disjoint global sections $s_0, s_{\infty}$ of this bundle such that $ev\circ s_0|_{(x,\mathcal{C}_x(X))}: \mathcal{C}(X) \rightarrow \mathcal{M}\rightarrow X$ collapses to one point $x\in X$ and image of $ev\circ s_\infty|_{(x,\mathcal{C}_x(X))}: \mathcal{C}(X) \rightarrow \mathcal{M}\rightarrow X$ is in the 'infinity part' if we choose some Harish-Chandra coordinate in $X$. From Lemma \ref{deletionasvmrt}, the 'infinity part' gives the embedding $X_0 \subset X$ as an admissible pair of deletion type.
Let $\mathcal{V}_x$ deonte the cone swept out by minimal rational curves marked at $x$. Then for each $t\in \mathbb{C}^*$, image of $ev\circ (s_0+ts_{\infty})|_{(x,\mathcal{C}_x(X))}$ is a linear section of $\mathcal{V}_x$ isomorphic to $\mathcal{C}_x(X)$ and image of $ev\circ (s_{\infty})|_{(x,\mathcal{C}_x(X))}$ is precisely $X_0$. 
Let $\mathcal{O}_{\mathcal{M}}(1)$ be the line bundle over $\mathcal{M}$ which restricted to each fiber $\mathbb{P}^1$ is $\mathcal{O}_{\mathbb{P}^1}(1)$. Then $ev_*\mathcal{O}_{\mathcal{M}}(1)=\mathcal{O}_X(1)$ as we can see that $ev_*\mathcal{O}_{\mathcal{M}}(1)$ restricted to each minimal rational curve in $X$ is isomorphic to $\mathcal{O}_{\mathbb{P}^1}(1)$. The variation of $t$ gives a section of the normal bundle $N_{X_0|X}$ which generates a line subbundle isomorphic to $ev_*\mathcal{O}_{\mathcal{M}}(1)|_{X_0}=\mathcal{O}_X(1)|_{X_0}=\mathcal{O}_{X_0}(1)$. Thus we can always find an $\mathcal{O}_{X_0}(1)$-factor in the decomposition of $N_{X_0|X}$. 

Moreover we can consider the normal bundle $\mathcal{N}_{\mathcal{C}(X)|\mathbb{P}T(X)}$ over $\mathcal{C}(X)$ which restricted to each $\mathcal{C}_x(X)$ is $N_{\mathcal{C}_x(X)|\mathbb{P}T_x(X)}$. Over each fibre, $N_{\mathcal{C}_x(X)|\mathbb{P}T_x(X)}$ is an irreducible $P_0$-homogeneous bundle. Let $ev': \mathcal{C}(X)\rightarrow X$ be the evaluation map, then $ev'_*(N_{\mathcal{C}_x(X)|\mathbb{P}T_x(X)})=N_{\mathcal{V}^{sm}_x|X}$ and hence $N_{X_0|X}=\mathcal{O}_{X_0}(1)\oplus N_{\mathcal{V}^{sm}_x|X}|_{X_0}$.

\end{proof}
\begin{remark}
		A natural way to prove the proposition is to use the branching rule in representation theory. In the special case in Lemma \ref{decomp}, we have such a geometric proof as above. In general, under the embedding $G_0/P_0\hookrightarrow G/P$ the normal bundle can be decomposed as direct sum of irreducible $P_0$-homogeneous bundles. We know the tangent bundle $T(G/P)$ is an irreducible $P$-homogeneous bundle, then the normal bundle decomposition can be obtained if we restrict the $P$-representation to the subgroup $P_0\subset P$, which can be described as a branching rule in representation theory. As the nilpotent parts act trivially we only need to consider the restriction of reductive parts of $P_0$ and $P$, denoted by $K_0, K$ respectively. The embedding of $K_0\subset K$ gives an embedding of Cartan subalgebras $\mathfrak{h}_0\subset \mathfrak{h}$ and their dual space $\mathfrak{h}^{\vee}_0\subset \mathfrak{h}^{\vee}$ from the root correspondence. After restriction of the representation, the operation on a weight in $\mathfrak{h}^{\vee}$ is the projection under $\mathfrak{h}^{\vee}=\mathfrak{h}^{\vee}_0\oplus Ann(\mathfrak{h}_0)$.
	
\end{remark}
Now we are ready to obtain the surjectivity of the restriction map $r: \Gamma(X, T_X)\rightarrow \Gamma(X_0, T_{X|X_0})$.

\begin{proposition}
	When $(X_0, X)$ is of deletion type and maximal, the restriction map
	$r: \Gamma(X, T_X)\rightarrow \Gamma(X_0, T_{X|X_0})$ is surjective.
\end{proposition}
 \begin{proof}
 The proof is similar to the proof in \cite{MR4028278}. As for hyperquadric case the conclusion is known in \cite{MR4028278}, we assume that $X$ is not a hyperquadric.
 It suffices to prove that $\Gamma(X, T_X)\rightarrow \Gamma(X_0, T_{X|X_0})\rightarrow \Gamma(X_0, N_{X_0|X})$ is surjective. From Proposition \ref{decomp} we know $N_{X_0|X}$ is a direct sum of $\mathcal{O}_{X_0}(1)$ and a $P_0$-homogeneous bundle, call it $\mathcal{E}$ . Since $X$ is not a hyperquadric, these two factors are not isomorphic to each other. Then from the Borel-Weil-Bott theorem, $\Gamma(X_0, N_{X_0|X})=\Gamma(X_0, \mathcal{O}_{X_0}(1))\oplus \Gamma(X_0, \mathcal{E})$ is a direct sum of non-isomorphic irreducible $G_0$-representation spaces.
 
 By the transitivity of $G$ on $X=G/P$, there exists a holomorphic one-parameter family $\{\gamma_t, t\in \Delta\}\subset G$ such that $\gamma_0=id_{G_0}$ and such that the projection of $\frac{d\gamma_t}{dt}|_{t=0} $ to $\Gamma(X_0, \mathcal{O}_{X_0}(1))$ is non-trivial. Also we can find such a family for $\Gamma(X_0, \mathcal{E})$. Then by Schur Lemma, since $\mathcal{E}$ and $\mathcal{O}_{X_0}(1)$ are not isomorphic, the restriction map is surjective.
 \end{proof}

\section{Counterexample violating the Condition (F)}\label{counterexample}

In this section we give a counterexample which violates the Fitting condition. We consider the embedding of $\{pt\}\cup \mathbb{P}^1$ in $G(2,3)\hookrightarrow \mathbb{P}^9$ as a linear section by $\mathbb{P}^2$ and $\mathbb{P}^1$ is embedded in $G(2,3)$ as a line. However in this case such linear sections can not be translated into each other by $Aut(G(2,3))=\mathbb{P}SL(5,\mathbb{C})$. Write the Pl\"ucker embedding as $p([\mathbb{C}u+\mathbb{C}v])=[u\wedge v]\in \mathbb{P}(\bigwedge^2\mathbb{C}^5)\cong \mathbb{P}^9 $ for  $\mathbb{C}u+\mathbb{C}v\subset \mathbb{C}^5$. Then the image is the projectivization of all decomposable vectors in $\bigwedge \mathbb{C}^5$. Let $e_1,\cdots, \
c_5$ be the basis of $\mathbb{C}^5$. We fix the line to be $\ell=\{[e_1\wedge (te_2+se_3)], [t,s]\in \mathbb{P}^1\}$. Standard choice of the point whose union with the line is a linear section of $G(2,3)$ is $[e_4\wedge e_5]$. We know the subgroup in $G$ stabilizes $\ell$ is a parabolic subgroup $Q$ associate to two nodes in the Dynkin diagram of $\mathfrak{g}=Lie(G)$.   We claim that there exists a point $a\in G(2,3)$ such that $\{a\}\cup \mathbb{P}^1$ is a linear section while $a$ can not be translated to $[e_4\wedge e_5]$ by $Q$. We know $Q$ can be expressed as

\[
\{C,C=\begin{bmatrix}
* & 0 & 0 &0 &0 \\
* & * & * &0 &0 \\
* & * & * &0 &0 \\
* & * & * &* &* \\
* & * & * &* &* \\
\end{bmatrix}\in \mathbb{P}SL(5,\mathbb{C})\}
\]
Then the $Q$-orbit on $[e_4\wedge e_5]$ is 
$\mathbb{A}=\{[(\sum_{1\leq i\leq 3}C_{4i}e_i+e_4) \wedge (\sum_{1
	\leq j\leq 3}C_{5j}e_j+ e_5)], C\in 
Q\}$, which is an affine cell in $G(2,3)$. 
We know the boundary divisor $\mathbb{D}=G(2,3)\backslash \mathbb{A}$ which is the hyperplane section of $G(2,3)$ given by $\{x_{45}=0\}\subset \mathbb{P}\bigwedge^2\mathbb{C}^5$ where $x_{ij}$ is the coordinate corresponding to $e_i\wedge e_j$. Let $N_2=\mathbb{P}(\mathbb{C}\{e_1\wedge e_2, e_2\wedge e_3, e_1\wedge e_3 \})(\cong \mathbb{P}^2)$. Then $\mathbb{D}=\{\text{lines in}\ G(2,3) \ \text{passing} \ N_2\}$. We know $\ell \subset N_2$, we then take a point $b\in \mathbb{D}\backslash \{\text{lines in}\ G(2,3) \ \text{passing} \ \ell\}$. We claim that $span<b,\ell>\cap G(2,3)=\{b\}\cap \ell$. If not, there is a point $d\in \ell$ such that $\overline{bd}\cap G(2,3)$ is at least 3 points, while $G(2,3)$ is defined by quadratic polynomials and $\overline{bd}\nsubseteq G(2,3)$, a contradiction. Hence we find such a $\mathbb{P}^2=span<b,\ell>$ whose intersection with $G(2,3)$ is a union of the point $b$ and the line $\ell$. However $b\notin \mathbb{A}$ is not in the $Q$-orbit on $e_4\wedge e_5$. Hence the embedding of $\{pt\}\cup \mathbb{P}^1$ does not satisfy the Fitting condition.
	\bibliographystyle{alpha}   
	\bibliography{research}
\end{document}